\documentclass[pdflatex,sn-mathphys-num]{sn-jnl}

\usepackage{amsfonts,amssymb,mathtools}
\usepackage{booktabs,array,microtype}
\usepackage{needspace,placeins}
\usepackage[nameinlink,capitalize,noabbrev]{cleveref}

\AddToHook{cmd/section/before}{\FloatBarrier\Needspace{7\baselineskip}}
\AddToHook{env/theorem/before}{\Needspace{9\baselineskip}}
\AddToHook{env/lemma/before}{\Needspace{7\baselineskip}}
\AddToHook{env/proposition/before}{\Needspace{9\baselineskip}}

\allowdisplaybreaks[2]
\numberwithin{equation}{section}
\theoremstyle{thmstyleone}
\newtheorem{theorem}{Theorem}[section]
\newtheorem{lemma}[theorem]{Lemma}
\newtheorem{proposition}[theorem]{Proposition}

\theoremstyle{thmstylethree}
\newtheorem{remark}[theorem]{Remark}

\crefname{theorem}{Theorem}{Theorems}
\Crefname{theorem}{Theorem}{Theorems}
\crefname{lemma}{Lemma}{Lemmas}
\Crefname{lemma}{Lemma}{Lemmas}
\crefname{proposition}{Proposition}{Propositions}
\Crefname{proposition}{Proposition}{Propositions}
\crefname{corollary}{Corollary}{Corollaries}
\Crefname{corollary}{Corollary}{Corollaries}
\crefname{remark}{Remark}{Remarks}
\Crefname{remark}{Remark}{Remarks}

\newcommand{\R}{\mathbb R}
\newcommand{\N}{\mathbb N}
\newcommand{\abs}[1]{\lvert#1\rvert}
\newcommand{\norm}[1]{\lVert#1\rVert}
\newcommand{\ip}[2]{\left\langle #1,#2\right\rangle}

\newcommand{\ind}{\mathbf 1}
\newcommand{\dd}{\,\mathrm d}
\newcommand{\wh}{\widehat}
\DeclareMathOperator{\diver}{div}

\newcommand{\revise}[1]{#1}

\title[An improved individual Li--Yau coefficient]
{%Toward P\'{o}lya's conjecture: 
Improved leading coefficient in the individual Berezin--Li--Yau bound via energy orthogonality}

\author[1]{\fnm{Yifan} \sur{Wang}}\email{wangyifan@lsec.cc.ac.cn}
\author*[2,3]{\fnm{Hehu} \sur{Xie}}\email{hhxie@lsec.cc.ac.cn}

\affil[1]{\orgdiv{School of Statistics and Mathematics},
  \orgname{Central University of Finance and Economics},
  \orgaddress{\city{Beijing}, \postcode{102206}, \country{China}}}

\affil*[2]{\orgdiv{SKLMS, NCMIS, Institute of Computational Mathematics},
  \orgname{Academy of Mathematics and Systems Science, Chinese Academy of Sciences},
  \orgaddress{\street{No.~55 Zhongguancun Donglu}, \city{Beijing},
  \postcode{100190}, \country{China}}}

\affil[3]{\orgdiv{School of Mathematical Sciences},
  \orgname{University of Chinese Academy of Sciences},
  \orgaddress{\city{Beijing}, \postcode{100049}, \country{China}}}

\begin{document}
\abstract{Let $\lambda_k$ be the $k$th eigenvalue of the Dirichlet
Laplacian on an open set $\Omega\subset\mathbb R^n$ of finite positive
measure.  The direct individual consequence of the Berezin--Li--Yau sum
inequality has leading coefficient $n/(n+2)$ relative to the Weyl term.
The main contribution of this paper is a strict improvement of this
coefficient. Energy orthogonality gives a frequency-dependent cap on
the Fourier density of the first $k$ eigenfunctions.  Combining this cap
with the standard Bessel bound and a bathtub principle for the radial
capacity
yields
\[
 \lambda_k\geq c_n(2\pi)^2\omega_n^{-2/n}
 |\Omega|^{-2/n}k^{2/n},
 \qquad \frac{n}{n+2}<c_n<1,
\]
for every $k\geq1$ and every $n\geq2$, without boundary regularity.
The constants are characterized by explicit scalar equations; in
dimension two, $c_2=0.5383068077\ldots$, giving a $7.66\%$ improvement
over the individual Li--Yau coefficient. We emphasize that this improves
the leading coefficient in the individual eigenvalue bound and the new constant $c_n$ 
is independent of the geometry and index $k$.} 
%not the sharp leading coefficient in the Li--Yau sum inequality. 
%We also compare the
%result with representative geometric, Riesz-mean, spectral-gap, and
%special-domain estimates.}



\maketitle

\tableofcontents
\clearpage

\section{Introduction}\label{sec:introduction}

Let $\Omega\subset\R^n$ be an open set with finite positive measure
$V=|\Omega|$, where $n\geq2$. We write
$\omega_n=|B_1(0)|$ and
\begin{equation}\label{eq:weyl-constant}
 C_n=(2\pi)^2\omega_n^{-2/n}.
\end{equation}
The eigenvalues of the Dirichlet Laplacian, listed in nondecreasing
order and repeated according to multiplicity, satisfy
\[
0<\lambda_1\leq\lambda_2\leq\cdots\nearrow\infty.
\]
Weyl's law gives the asymptotic formula (see, for example,
\cite{SafarovVassiliev1997})
\[
\lambda_k\sim C_nV^{-2/n}k^{2/n}
\qquad\text{as }k\to\infty.
\]
P\'olya's conjecture~\cite{Polya1961} asks whether the Weyl term is a
termwise lower bound,
\begin{equation}\label{eq:polya}
 \lambda_k\geq C_nV^{-2/n}k^{2/n},
\end{equation}
for every $k$. The conjecture remains open for general domains. The
classical Berezin~\cite{Berezin1972} and Li--Yau inequality~\cite{LiYau1983} gives
\begin{equation}\label{eq:li-yau-sum}
 \sum_{j=1}^k\lambda_j
 \geq \frac{n}{n+2}C_nV^{-2/n}k^{1+2/n},
\end{equation}
and the monotonicity of the spectrum yields the individual consequence
\begin{equation}\label{eq:li-yau-individual}
 \lambda_k\geq \frac{n}{n+2}C_nV^{-2/n}k^{2/n}.
\end{equation}

The distinction between \cref{eq:li-yau-sum,eq:li-yau-individual} is
central to this paper.  The coefficient $n/(n+2)$ is asymptotically 
sharp in the Li--Yau \emph{sum} inequality, but the passage from the sum to the
individual estimate uses only
$\sum_{j=1}^k\lambda_j\leq k\lambda_k$ and loses spectral information.
Our main result improves the coefficient of the leading Weyl-scale term
in the resulting \emph{individual} bound:
\begin{equation}\label{eq:intro-main-result}
 \lambda_k\geq c_nC_nV^{-2/n}k^{2/n},
 \qquad \frac{n}{n+2}<c_n<1.
\end{equation}
Thus the principal term retains the correct dependence on $V$ and $k$
while its universal coefficient is strictly larger than the direct
Li--Yau coefficient.  This leading-coefficient improvement is the main
contribution of the article.

The mechanism is a second Fourier-space constraint.  The usual
$L^2$ Bessel inequality bounds the spectral density by a constant.
After energy normalization, the gradients of Dirichlet eigenfunctions
form an orthonormal family.  Writing a plane wave as the divergence of
an explicit vector field and integrating weakly by parts gives an
additional cap that decays like $|\xi|^{-2}$.  Imposing both caps before
minimizing the Fourier second moment produces a strictly stronger
individual estimate.  The resulting variational problem is radial and
reduces to an explicit scalar equation.
The two-dimensional case is treated separately because the radial mass
of the $|\xi|^{-2}$ tail is logarithmic.

The result is universal: no smoothness, convexity, tiling property, or
connectedness assumption is imposed on $\Omega$.  It does not prove
P\'{o}lya's conjecture, since $c_n<1$. 
%and it does not improve the sharp
%leading coefficient in \cref{eq:li-yau-sum}. 
These distinctions are important when comparing the theorem with geometric remainder terms,
Riesz-mean estimates, spectral-gap inequalities, and coefficient-one
results available for special domains.  Possible refinements include
the inertia corrections of Melas~\cite{Melas2003} and
Ilyin~\cite{Ilyin2010}, boundary-sensitive estimates
\cite{KVW2009,GLW2011}, recent Riesz-mean and gap improvements
\cite{FrankLarson2025,FrankLarson2026,FLP2025,Steinerberger2024}, and results related to
P\'{o}lya's conjecture on subsequences or special geometries
\cite{GJL2025,FLPSBalls2023,HeWang2026,JiangLin2026}.

The paper is organized as follows.  \Cref{sec:preliminaries} records the
Fourier identities and the radial-capacity bathtub principle.
\Cref{sec:envelope} derives the improved Fourier envelope from energy
orthogonality.  \Cref{sec:main-results} proves
\cref{eq:intro-main-result}, treats the two-dimensional logarithmic
case, and compares the constants numerically.  %\Cref{sec:comparison}
%explains precisely which leading term is improved and places the result
%among representative estimates in the literature.
\Cref{sec:questions} summarizes the conclusions and discusses possible
routes toward the P\'{o}lya coefficient.

\section{Preliminaries}\label{sec:preliminaries}

We begin by recording the spectral and Fourier identities on which the
argument rests. Let $\{\phi_j\}_{j\geq1}$ be a real-valued orthonormal basis of
Dirichlet eigenfunctions in $L^2(\Omega)$. We use complex Hilbert spaces
when pairing these functions with plane waves, with inner product linear
in its first argument. Thus
\begin{equation}\label{eq:eigenproblem}
-\Delta\phi_j=\lambda_j\phi_j,
\qquad \phi_j\in H_0^1(\Omega),
\qquad \int_\Omega\phi_i\phi_j\dd x=\delta_{ij}.
\end{equation}
We use the unitary Fourier transform
\begin{equation}\label{eq:fourier}
\wh u(\xi)=(2\pi)^{-\frac{n}{2}}
\int_{\R^n}u(x)e^{-\mathrm{i}x\cdot\xi}\dd x,
\end{equation}
where every function in $H_0^1(\Omega)$ is extended by zero to $\R^n$. The
Fourier density associated with the first $k$ eigenfunctions is
\begin{equation}\label{eq:density}
F_k(\xi)=\sum_{j=1}^k\abs{\wh\phi_j(\xi)}^2.
\end{equation}
Plancherel's theorem and the eigenvalue equation then give the two basic
identities
\begin{equation}\label{eq:mass-moment}
\int_{\R^n}F_k(\xi)\dd\xi=k,
\qquad
\int_{\R^n}\abs\xi^2F_k(\xi)\dd\xi
=\sum_{j=1}^k\lambda_j.
\end{equation}
Thus the spectral problem has been converted into a constrained moment problem:
the total mass of $F_k$ is known exactly, its second moment is the eigenvalue
sum, and the remaining task is to obtain the sharpest available pointwise
majorant.

Before deriving that majorant, we formulate the rearrangement principle that
will turn it into a lower bound for the second moment. The following statement
is a radial-obstacle version of the standard bathtub principle; see
Lieb and Loss~\cite{LiebLoss2001}.

\begin{lemma}[Bathtub principle under a radial envelope]\label[lemma]{lem:bathtub}
Let $h:[0,\infty)\to[0,\infty)$ be measurable and nonincreasing, and suppose
that $h(\abs\cdot)\in L^1_{\mathrm{loc}}(\R^n)$. Let $0\leq M<\infty$ and
assume that
\[
\int_{\R^n}h(\abs\xi)\dd\xi\geq M.
\]
If a measurable function $f:\R^n\to[0,\infty)$ satisfies
\[
0\leq f(\xi)\leq h(\abs\xi)\quad\text{a.e.},
\qquad
\int_{\R^n}f(\xi)\dd\xi=M,
\]
then there exists $R\in[0,\infty]$ such that
\[
\int_{B_R}h(\abs\xi)\dd\xi=M.
\]
Consequently, if
\[
g_R(\xi):=h(\abs\xi)\ind_{B_R}(\xi),
\]
then
\begin{equation}\label{eq:bathtub}
\int_{\R^n}\abs\xi^2f(\xi)\dd\xi
\geq
\int_{B_R}\abs\xi^2h(\abs\xi)\dd\xi.
\end{equation}
The integrals in \cref{eq:bathtub} may take the value $+\infty$.
\end{lemma}

\begin{proof}
The proof consists of three steps.

\medskip
\noindent\emph{Step 1: determination of the filling radius.}
Define the cumulative mass
\[
H(r):=\int_{B_r}h(\abs\xi)\dd\xi,
\qquad r\in[0,\infty).
\]
Since $h(\abs\cdot)\in L^1_{\mathrm{loc}}(\R^n)$, one has $H(r)<\infty$
for every finite $r$. Polar coordinates give
\begin{equation}\label{eq:H-polar}
H(r)=n\omega_n\int_0^r h(s)s^{n-1}\dd s.
\end{equation}
It follows that $H$ is absolutely continuous on every finite interval; in
particular, it is continuous and nondecreasing. Moreover, the monotone
convergence theorem yields
\[
\lim_{r\to\infty}H(r)
=\int_{\R^n}h(\abs\xi)\dd\xi\geq M.
\]
If $M=0$, we take $R=0$. Suppose henceforth that $M>0$. If $H(R)=M$ for some
finite $R$, we choose such an $R$. Otherwise, continuity forces
\[
H(r)<M\quad\text{for every finite }r,
\qquad
\lim_{r\to\infty}H(r)=M,
\]
and we set $R=\infty$. In either case,
\begin{equation}\label{eq:equal-mass}
\int_{\R^n}g_R(\xi)\dd\xi
=\int_{B_R}h(\abs\xi)\dd\xi=M
=\int_{\R^n}f(\xi)\dd\xi.
\end{equation}

\medskip
\noindent\emph{Step 2: comparison when $R<\infty$.}
Inside $B_R$, the bound $f\leq h$ implies
\[
f(\xi)-g_R(\xi)=f(\xi)-h(\abs\xi)\leq0,
\qquad
\abs\xi^2-R^2\leq0.
\]
Outside $B_R$, one has $g_R=0$, and hence
\[
f(\xi)-g_R(\xi)=f(\xi)\geq0,
\qquad
\abs\xi^2-R^2\geq0.
\]
Combining the two regions gives the pointwise sign relation
\begin{equation}\label{eq:pointwise-sign}
\bigl(\abs\xi^2-R^2\bigr)
\bigl(f(\xi)-g_R(\xi)\bigr)\geq0
\quad\text{a.e. in }\R^n.
\end{equation}
Furthermore, local integrability of the envelope implies
\[
\int_{\R^n}\abs\xi^2g_R(\xi)\dd\xi
\leq R^2\int_{B_R}h(\abs\xi)\dd\xi=R^2M<\infty.
\]
If the second moment of $f$ is infinite, the desired conclusion is immediate.
Otherwise, both second moments are finite, and integration of
\cref{eq:pointwise-sign}, together with the equal-mass identity
\cref{eq:equal-mass}, yields
\begin{align*}
0
&\leq\int_{\R^n}
\bigl(\abs\xi^2-R^2\bigr)
\bigl(f(\xi)-g_R(\xi)\bigr)\dd\xi\\
&=\int_{\R^n}\abs\xi^2
\bigl(f(\xi)-g_R(\xi)\bigr)\dd\xi
-R^2\int_{\R^n}\bigl(f(\xi)-g_R(\xi)\bigr)\dd\xi\\
&=\int_{\R^n}\abs\xi^2f(\xi)\dd\xi
-\int_{\R^n}\abs\xi^2g_R(\xi)\dd\xi.
\end{align*}
This is precisely \cref{eq:bathtub}.

\medskip
\noindent\emph{Step 3: the case $R=\infty$.}
In this case $g_R(\xi)=h(\abs\xi)$, while \cref{eq:equal-mass} gives
\[
0\leq h(\abs\xi)-f(\xi),
\qquad
\int_{\R^n}\bigl(h(\abs\xi)-f(\xi)\bigr)\dd\xi=0.
\]
It follows that $f(\xi)=h(\abs\xi)$ almost everywhere. Hence
\cref{eq:bathtub} holds with equality, completing the proof.
\end{proof}

\begin{remark}[Partial filling on the boundary sphere]\label{rem:sphere}
For Lebesgue measure, the sphere $\partial B_R$ has zero $n$-dimensional
measure. Changing the value of $f$ on the sphere therefore neither changes nor
adjusts its mass.
\revise{
If the prescribed mass is strictly smaller than the total
capacity, continuity of $H$ guarantees a finite filling radius
with $H(R)=M$, without any boundary filling. If the prescribed
mass equals the total capacity, the filling radius may be
infinite, as allowed in Lemma~\ref{lem:bathtub}.
In more general bathtub problems, partial filling may be needed
when a level set of the cost function has positive measure
with respect to the capacity measure.
}

\end{remark}

\begin{remark}[Equality]\label{rem:equality}
If $R<\infty$, equality in \cref{eq:bathtub} requires
\[
\bigl(\abs\xi^2-R^2\bigr)
\bigl(f(\xi)-g_R(\xi)\bigr)=0
\quad\text{a.e.}
\]
Since $\partial B_R$ has measure zero, this forces
\[
f(\xi)=h(\abs\xi)\quad\text{a.e. in }B_R,
\qquad
f(\xi)=0\quad\text{a.e. in }\R^n\setminus B_R.
\]
Thus, among all densities of mass $M$ lying below the prescribed capacity
$h(\abs\xi)$, the second moment is minimized by filling the available capacity
from the frequency origin outward.
\end{remark}

The lemma isolates the variational part of the argument. What remains is to
identify a useful radial envelope for $F_k$. The usual $L^2$ argument supplies
a constant cap, whereas the energy inner product supplies a second cap that
decays at high frequency. Their minimum is the crucial input in the next
section.

\section{The Frequency Envelope from Energy Orthogonality}
\label{sec:envelope}

We now derive the main Fourier-space estimate. The first half is the familiar
$L^2$ Bessel bound. The second half uses the Dirichlet condition through weak
integration by parts and is responsible for the improvement.

\begin{proposition}[Improved Fourier projection estimate]\label[proposition]{prop:envelope}
For every $k\in\N$ and almost every $\xi\in\R^n$,
\begin{equation}\label{eq:projection-envelope}
F_k(\xi)
\leq \frac{V}{(2\pi)^n}
\min\left\{1,\frac{\lambda_k}{\abs\xi^2}\right\},
\end{equation}
where only the first term is used at $\xi=0$.
\end{proposition}

\begin{proof}
Regard $(2\pi)^{-n/2}e^{\mathrm{i}x\cdot\xi}$ as an element of
$L^2(\Omega)$. Bessel's inequality gives \cite{LiYau1983}
\begin{equation}\label{eq:l2-bessel}
F_k(\xi)\leq(2\pi)^{-n}\int_\Omega1\dd x
=\frac{V}{(2\pi)^n}.
\end{equation}

To obtain the second bound, we use the weak formulation of
\cref{eq:eigenproblem}. The vector fields
\begin{equation}\label{eq:energy-orthogonal}
e_j=\frac{\nabla\phi_j}{\sqrt{\lambda_j}},
\qquad j\geq1,
\end{equation}
form an orthonormal family in $L^2(\Omega;\mathbb C^n)$. Fix $\xi\neq0$ and define
\[
X_\xi(x)=\frac{\mathrm{i}\xi}{\abs\xi^2}e^{-\mathrm{i}x\cdot\xi};
\qquad
\diver X_\xi=e^{-\mathrm{i}x\cdot\xi}.
\]
Because $\phi_j\in H_0^1(\Omega)$, weak integration by parts gives
\[
\wh\phi_j(\xi)
=-(2\pi)^{-\frac n2}\int_\Omega
\nabla\phi_j(x)\cdot X_\xi(x)\dd x.
\]
Consequently,
\begin{align*}
F_k(\xi)
&=(2\pi)^{-n}\sum_{j=1}^k
\lambda_j\left|\ip{e_j}{\overline{X_\xi}}_{L^2(\Omega;\mathbb C^n)}\right|^2\\
&\leq(2\pi)^{-n}\lambda_k
\sum_{j=1}^k\left|\ip{e_j}{\overline{X_\xi}}\right|^2\\
&\leq(2\pi)^{-n}\lambda_k\norm{X_\xi}_{L^2(\Omega)}^2
=\frac{V}{(2\pi)^n}\frac{\lambda_k}{\abs\xi^2},
\end{align*}
where the last inequality is the vector-valued $L^2$ Bessel inequality.
Combining this estimate with \cref{eq:l2-bessel} proves
\cref{eq:projection-envelope}.
\end{proof}

\begin{remark}\label{rem:source-improvement}
The constant bound \cref{eq:l2-bessel}, together with
\cref{eq:mass-moment}, recovers the classical Li--Yau inequality. The additional
constraint is active in the region
$\abs\xi>\sqrt{\lambda_k}$.
\revise{When the classical filling ball extends beyond
$\abs{\xi}=\sqrt{\lambda_k}$, the additional cap reduces the
available capacity there and forces the
second-moment-minimizing density to occupy a larger ball.}
\end{remark}

\begin{remark}[A zero-mean refinement]\label{rem:zero-mean-refinement}
The energy estimate admits a further refinement. Since
$e_j=\nabla\phi_j/\sqrt{\lambda_j}$ and $\phi_j\in H_0^1(\Omega)$,
approximation by compactly supported smooth functions gives
\[
 \int_\Omega e_j(x)\dd x=0.
\]
For $\xi\ne0$, define
\[
 X_\xi(x)=\frac{\mathrm{i}\xi}{|\xi|^2}e^{-\mathrm{i}x\cdot\xi},
 \qquad
 Z_\xi=X_\xi-\frac1V\int_\Omega X_\xi\dd x,
 \qquad
 m_\Omega(\xi)=\int_\Omega e^{-\mathrm{i}x\cdot\xi}\dd x.
\]
Weak integration by parts and the zero-mean identity give
\[
 \wh\phi_j(\xi)
 =-(2\pi)^{-n/2}\sqrt{\lambda_j}\int_\Omega e_j\cdot X_\xi\dd x
 =-(2\pi)^{-n/2}\sqrt{\lambda_j}\int_\Omega e_j\cdot Z_\xi\dd x.
\]
The last integral is
$\langle e_j,\overline{Z_\xi}\rangle_{L^2(\Omega;\mathbb C^n)}$.
The vector-valued Bessel inequality therefore yields
\begin{align*}
 F_k(\xi)
 &\leq(2\pi)^{-n}\lambda_k
       \sum_{j=1}^k
       \left|\left\langle e_j,\overline{Z_\xi}\right\rangle\right|^2\\
 &\leq(2\pi)^{-n}\lambda_k\|Z_\xi\|_{L^2(\Omega)}^2\\
 &=\frac{V}{(2\pi)^n}\frac{\lambda_k}{|\xi|^2}
   \left(1-\frac{|m_\Omega(\xi)|^2}{V^2}\right).
\end{align*}
Consequently, \cref{eq:projection-envelope} can be sharpened to
\begin{equation}\label{eq:projection-envelope-2}
 F_k(\xi)\leq \frac{V}{(2\pi)^n}
 \min\left\{1,\frac{\lambda_k}{|\xi|^2}
 \left(1-\frac{|m_\Omega(\xi)|^2}{V^2}\right)\right\}.
\end{equation}
This capacity is generally nonradial. The elementary volume-only
theorems below deliberately retain the simpler radial cap.  
With the help of geometric information, we can obtain a sharper lower bound for 
indival eigenvalue.
%In \cref{sec:inertia-envelope} we return to
%\cref{eq:projection-envelope-2}, bound its domain Fourier factor through
%the centered inertia tensor, and obtain a computable radial extension.
\end{remark}

\begin{remark}
The orthogonal family \(e_j\) is adopted in our numerical analysis for 
the augmented subspace method \cite{DangWangXieZhou2023,XieZhangOwhadi2019}. 
%numerical eigenvalue problems, for instance, in extended subspace methods. 
This also motivates our introduction of the energy orthogonality technique here.
%The family of $e_j$ in (\ref{eq:energy-orthogonal}) is used in the 
%analysis for our proposed augmented subspace for numerically solving 
%eigenvalue problems \cite{DangWangXieZhou2023,XieZhangOwhadi2019}. 
%The treatment in Remark \ref{rem:zero-mean-refinement} 
%also has the characteristic of numerical theory. 
\end{remark}
We are now in a position to combine the two ingredients. The projection
estimate supplies a radial obstacle depending on $L=\lambda_k$, and
\cref{lem:bathtub} identifies the density below that obstacle with the smallest
possible second moment. The final closure is provided by the elementary upper
bound $\sum_{j=1}^k\lambda_j\leq k\lambda_k$. Because the radial mass of an
$r^{-2}$ tail changes from a power law to a logarithm at $n=2$, the cases
$n\geq3$ and $n=2$ must be treated separately.

\section{Volume-Only Eigenvalue Bounds}\label{sec:main-results}

We now prove the main result: a strict improvement of the coefficient of
the Weyl-scale leading term in the individual Li--Yau bound.  The moment
estimate is closed using only
$\sum_{j\leq k}\lambda_j\leq k\lambda_k$, so the conclusion applies to
every open set of finite positive measure. Set
\begin{equation}\label{eq:alpha}
\alpha_n=\frac{\omega_nV}{(2\pi)^n}.
\end{equation}

\subsection{Higher dimensions}

For $n\geq3$, let $x_n>1$ be the unique solution in $(1,\infty)$ of
\begin{equation}\label{eq:xn}
(n-2)x^n-nx^{n-2}+\frac{8}{n+2}=0,
\end{equation}
and define
\begin{equation}\label{eq:Kn-cn}
K_n=\frac{nx_n^{n-2}-2}{n-2},
\qquad
c_n=K_n^{-\frac{2}{n}}.
\end{equation}
These constants arise naturally when the lower bound for the Fourier second
moment is matched with its spectral upper bound. The resulting estimate is the
following.

\begin{theorem}[Universal improvement in higher dimensions]
\label{thm:main-high}
Let $\Omega\subset\R^n$ be an open set of finite positive measure and suppose that
$n\geq3$. Then, for every $k\geq1$,
\begin{equation}\label{eq:main-high}
\lambda_k\geq c_nC_nV^{-\frac{2}{n}}k^{\frac{2}{n}}.
\end{equation}
Moreover, the computable constant $c_n$ satisfies 
\begin{equation}\label{eq:strict-comparison}
\frac{n}{n+2}<c_n<1.
\end{equation}
\end{theorem}

\begin{proof}
Write $L=\lambda_k$ and introduce the dimensionless parameter
\begin{equation}\label{eq:rho}
\rho=\frac{k}{\alpha_nL^{\frac n2}}.
\end{equation}
If $\rho\leq1$, then rearrangement of this inequality immediately gives the
P\'olya bound \cref{eq:polya}, which is stronger than
\cref{eq:main-high}. We may therefore assume that $\rho>1$.

In terms of $L$, the envelope in \cref{prop:envelope} is
\[
h_L(r)=\frac{V}{(2\pi)^n}\min\left\{1,\frac{L}{r^2}\right\}.
\]
By \cref{lem:bathtub}, the admissible density of mass $k$ with the smallest
second moment fills this envelope on a ball $B_R$. Put $R=x\sqrt L$.
The condition $\rho>1$ forces $x>1$, so the filling ball extends into the
decaying part of the envelope. Direct radial integration then yields
\begin{align}
k&=\alpha_nL^{\frac n2}
\frac{nx^{n-2}-2}{n-2},\label{eq:mass-high}\\
\sum_{j=1}^k\lambda_j
&\geq\alpha_nL^{1+\frac n2}
\left(x^n-\frac{2}{n+2}\right).
\label{eq:moment-high}
\end{align}
On the other hand, monotonicity of the eigenvalues gives the complementary
upper bound
\begin{equation}\label{eq:upper-moment}
\sum_{j=1}^k\lambda_j\leq k\lambda_k=kL.
\end{equation}
Substituting \cref{eq:mass-high} into
\cref{eq:moment-high,eq:upper-moment} gives
\begin{equation}\label{eq:root-inequality}
(n-2)x^n-nx^{n-2}+\frac{8}{n+2}\leq0.
\end{equation}
The derivative of the left-hand side is
\[
n(n-2)x^{n-3}(x^2-1)>0
\qquad (x>1).
\]
The same function is negative at $x=1$ and tends to $+\infty$ as
$x\to\infty$. Hence it has exactly one zero $x_n$ in $(1,\infty)$, and
\cref{eq:root-inequality} implies $x\leq x_n$. Using
\cref{eq:mass-high}, we therefore obtain
\[
\rho=\frac{nx^{n-2}-2}{n-2}\leq K_n.
\]
Returning to \cref{eq:rho} and using \cref{eq:weyl-constant}, we conclude that
\[
L\geq\left(\frac{k}{\alpha_nK_n}\right)^{\frac2n}
=c_nC_nV^{-\frac2n}k^{\frac2n}.
\]

It remains to locate the coefficient relative to the Li--Yau and P\'olya
constants. Since $K_n>1$, one immediately has $c_n<1$. To see that the
improvement over Li--Yau is strict without resorting to a lengthy algebraic
comparison, suppose that $L$ equals the right-hand side of
\cref{eq:li-yau-individual}. For the classical constant envelope, the bathtub
lower bound then equals $kL$. At that value of $L$,
\[
\rho=\left(\frac{n+2}{n}\right)^{\frac{n}{2}}>1,
\]
so the minimizing ball necessarily reaches the region
$\abs\xi>\sqrt L$. There the new envelope is strictly below the classical
constant one. Consequently, the smallest second moment of a density with mass
$k$ is strictly greater than $kL$, contradicting
\cref{eq:upper-moment}. Continuity now gives $c_n>n/(n+2)$ and completes the
proof.
\end{proof}

\subsection{The critical two-dimensional case}

When $n=2$, the radial mass of the $r^{-2}$ tail grows logarithmically, so the
power-law calculation above has to be replaced by a separate computation. Let
$y_2>1$ be the unique solution in $(1,\infty)$ of
\begin{equation}\label{eq:y2}
y-\log y=\frac32,
\end{equation}
and define
\begin{equation}\label{eq:c2}
K_2=1+\log y_2=y_2-\frac12,
\qquad
c_2=K_2^{-1}.
\end{equation}

\begin{theorem}[Universal improvement in dimension two]
\label{thm:main-2d}
Let $\Omega\subset\R^2$ be an open set of finite positive measure. Then, for every
$k\geq1$,
\begin{equation}\label{eq:main-2d}
\lambda_k\geq c_2\frac{4\pi k}{V},
\qquad c_2=0.5383068077\ldots.
\end{equation}
In particular, the coefficient improves the individual Li--Yau bound
$\lambda_k\geq2\pi k/V$ by approximately $7.66\%$.
\end{theorem}

\begin{proof}
Again write $L=\lambda_k$, and note that $\alpha_2=V/(4\pi)$. If
$k\leq\alpha_2L$, then the P\'olya bound already follows. Otherwise, write the
saturation radius as $R^2=yL$ with $y>1$. Radial integration gives
\begin{align}
k&=\alpha_2L(1+\log y),\label{eq:mass-2d}\\
\sum_{j=1}^k\lambda_j
&\geq\alpha_2L^2\left(y-\frac12\right).
\label{eq:moment-2d}
\end{align}
Combining \cref{eq:upper-moment,eq:mass-2d,eq:moment-2d}, we obtain
\[
y-\frac12\leq1+\log y,
\qquad\text{or equivalently}\qquad
y-\log y\leq\frac32.
\]
The function $y-\log y$ is strictly increasing for $y>1$. Hence
$y\leq y_2$, and therefore
\[
\frac{k}{\alpha_2L}=1+\log y\leq K_2.
\]
Rearranging proves \cref{eq:main-2d}. The stated numerical value follows from
\cref{eq:y2,eq:c2}.
\end{proof}

\subsection{Numerical comparison of the constants}

To illustrate the size of the improvement, \cref{tab:constants} lists the
coefficients in several dimensions. The relative gain is defined as
\[
\left(\frac{c_n}{n/(n+2)}-1\right)\times100\%.
\]
The displayed gains decrease across the sampled dimensions. Strict
positivity in every dimension $n\geq2$ follows from the preceding theorems.

\begin{table}[ht]
\centering
\caption{Comparison between the coefficient obtained here and the Li--Yau
coefficient.}\label{tab:constants}
\begin{tabular}{@{}cccc@{}}
\toprule
$n$ & Li--Yau coefficient $n/(n+2)$ & Present coefficient $c_n$
& Relative gain \\
\midrule
2  & 0.500000 & 0.538307 & 7.66\% \\
3  & 0.600000 & 0.622248 & 3.71\% \\
4  & 0.666667 & 0.681250 & 2.19\% \\
5  & 0.714286 & 0.724595 & 1.44\% \\
6  & 0.750000 & 0.757678 & 1.02\% \\
8  & 0.800000 & 0.804735 & 0.59\% \\
10 & 0.833333 & 0.836544 & 0.39\% \\
\bottomrule
\end{tabular}
\end{table}

\begin{remark}[Interpretation of the main contribution]
The improvement in \cref{thm:main-high,thm:main-2d} is a leading-term
improvement for an \emph{individual} eigenvalue: the universal
coefficient multiplying $C_nV^{-2/n}k^{2/n}$ increases from
$n/(n+2)$ to $c_n$.  The power of $k$ and the volume scaling are
unchanged.  
This statement must not be confused with an improvement of
the leading coefficient in the Li--Yau sum inequality
\cref{eq:li-yau-sum}, whose coefficient is already asymptotically sharp.
\end{remark}

\section{Conclusions and Further Directions}\label{sec:questions}

The principal result of this paper is the universal estimate
\[
 \lambda_k\geq c_nC_nV^{-2/n}k^{2/n},
 \qquad \frac{n}{n+2}<c_n<1,
\]
valid for every index on every finite-measure open set in dimension
$n\geq2$.  Its main novelty is the strict improvement from $n/(n+2)$
to $c_n$ in the leading coefficient of the individual Li--Yau bound.
The Weyl-scale power of $k$ and the volume dependence are preserved, and
no boundary regularity is required.

The improvement comes from retaining energy orthogonality in Fourier
space.  In addition to the classical constant Bessel cap, the Dirichlet
condition produces a cap proportional to
$\lambda_k|\xi|^{-2}$.  Their minimum reduces the high-frequency
capacity available to the Fourier density; the bathtub principle then
forces the mass-minimizing density outward and raises its second moment.
This mechanism is elementary, scale invariant, and distinct from
lower-order geometric remainder terms.

%The scope of the result is equally important. 
%The coefficient in the
%Li--Yau sum inequality is already sharp and is not improved here.
%Moreover, $c_n<1$, so the theorem does not reach the P\'{o}lya
%coefficient.
Recently, it is well known that the coefficient-one bounds on balls \cite{FLPSBalls2023}, 
thin products \cite{HeWang2026}, 
or other special classes remain stronger under their hypotheses, 
while boundary-sensitive and Riesz-mean estimates retain information that a
volume-only theorem cannot encode. 
A natural and particularly promising direction is to combine the
present Fourier envelope with additional geometric information,
counting-function techniques, and upper bounds for individual
eigenvalues or eigenvalue sums. Under the corresponding geometric or
regularity assumptions, these inputs are likely to sharpen the resulting
eigenvalue lower bounds. Representative sources of such information
include geometric and boundary corrections
\cite{Melas2003,Ilyin2010,KVW2009,GLW2011}, counting-function and Weyl
remainder estimates~\cite{FrankLarson2025,FrankLarson2026,GJL2025,JiangLin2026}, 
spectral-gap information~\cite{Steinerberger2024}, and suitable upper bounds for
individual eigenvalues or their sums. Other possible extensions include
nonradial Fourier constraints and quantitative orthogonality beyond a
single energy level. Reaching the P\'{o}lya
coefficient through this framework will ultimately require a genuinely
stronger principal-scale restriction, rather than only a lower-order
correction. The techniques here also has the potential application to 
other types of eigenvalue problems. 

\backmatter

\section*{Statements and Declarations}

\noindent\textbf{Funding.}
%The authors should insert the applicable funding information before
%submission.
This work was supported by the National Key Research and Development 
Program of China (2025YFA1016600, 2025YFA1016601),  
National Natural Science Foundations of China (12331015), 
and National Center for Mathematics and Interdisciplinary Science, CAS.

\noindent\textbf{Competing interests.}
The authors declare that there is no conflict of interest. 

\noindent\textbf{Data availability.}
This article is theoretical and does not use external datasets.

%\noindent\textbf{Code availability.}
%No research code is required for the proofs.

%\Needspace{5\baselineskip}
%\noindent\textbf{Author contributions.}
%Yifan Wang contributed to the discussion and verification of the
%derivations and to the joint writing of the manuscript. Hehu Xie
%proposed the original method, designed the overall strategy, and wrote
%and revised the initial draft of the manuscript.


\begin{thebibliography}{99}

\bibitem{Berezin1972}
F.~A. Berezin,
Covariant and contravariant symbols of operators,
\emph{Izv. Akad. Nauk SSSR Ser. Mat.} \textbf{36} (1972), 1134--1167.
\href{https://doi.org/10.1070/IM1972V006N05ABEH001913}{doi:10.1070/IM1972V006N05ABEH001913}.


\bibitem{FLP2025}
R.~L. Frank, S. Larson and P. Pfeiffer,
Improved semiclassical eigenvalue estimates for the Laplacian and the Landau Hamiltonian,
\emph{J. Spectr. Theory} \textbf{16} (2026), no.~1, 243--270.
\href{https://doi.org/10.4171/JST/589}{doi:10.4171/JST/589}.

\bibitem{FrankLarson2025}
R.~L. Frank and S. Larson,
Riesz means asymptotics for Dirichlet and Neumann Laplacians on
Lipschitz domains,
\emph{Invent. Math.} \textbf{241} (2025), 999--1079.
\href{https://doi.org/10.1007/s00222-025-01352-x}%
{doi:10.1007/s00222-025-\allowbreak\mbox{01352-x}}.


\bibitem{FrankLarson2026}
R.~L. Frank and S. Larson,
Semiclassical inequalities for Dirichlet and Neumann Laplacians on convex domains,
\emph{Comm. Pure Appl. Math.} \textbf{79} (2026), no.~3, 762--822.
\href{https://doi.org/10.1002/cpa.70019}{doi:10.1002/cpa.70019}.
%Preprint:
%\href{https://arxiv.org/abs/2410.04769}{arXiv:2410.04769}.

\bibitem{GJL2025}
Z. Gan, R. Jiang and F. Lin,
Improved Berezin--Li--Yau inequality and Kr\"oger inequality and consequences,
\href{https://arxiv.org/abs/2507.20330}{arXiv:2507.20330}, 2025.

\bibitem{GLW2011}
L. Geisinger, A. Laptev and T. Weidl,
Geometrical versions of improved Berezin--Li--Yau inequalities,
\emph{J. Spectr. Theory} \textbf{1} (2011), 87--109.
\href{https://arxiv.org/abs/1010.2683}{arXiv:1010.2683}.

\bibitem{Ilyin2010}
A.~A. Ilyin,
Lower bounds for the spectrum of the Laplace and Stokes operators,
\emph{Discrete Contin. Dyn. Syst.} \textbf{28} (2010), 131--146.
\href{https://arxiv.org/abs/0909.2818}{arXiv:0909.2818}.

\bibitem{KVW2009}
H. Kova\v{r}\'{\i}k, S. Vugalter and T. Weidl,
Two-dimensional Berezin--Li--Yau inequalities with a correction term,
\emph{Commun. Math. Phys.} \textbf{287} (2009), 959--981.
\href{https://arxiv.org/abs/0802.2792}{arXiv:0802.2792}.

\bibitem{LiYau1983}
P. Li and S.-T. Yau,
On the Schr\"odinger equation and the eigenvalue problem,
\emph{Commun. Math. Phys.} \textbf{88} (1983), 309--318.
\href{https://doi.org/10.1007/BF01213210}{doi:10.1007/BF01213210}.

\bibitem{LiebLoss2001}
E.~H. Lieb and M. Loss,
\emph{Analysis}, 2nd ed.,
Graduate Studies in Mathematics, Vol.~14, American Mathematical Society, 2001.

\bibitem{Melas2003}
A.~D. Melas,
A lower bound for sums of eigenvalues of the Laplacian,
\emph{Proc. Amer. Math. Soc.} \textbf{131} (2003), 631--636.

\bibitem{Polya1961}
G. P\'olya,
On the eigenvalues of vibrating membranes,
\emph{Proc. London Math. Soc.} \textbf{11} (1961), 419--433.

\bibitem{SafarovVassiliev1997}
Yu. Safarov and D. Vassiliev,
\emph{The Asymptotic Distribution of Eigenvalues of Partial Differential
Operators},
Translations of Mathematical Monographs, vol.~155,
American Mathematical Society, Providence, RI, 1997.

\bibitem{Steinerberger2024}
{\raggedright S. Steinerberger,
Universal lower bounds for Dirichlet eigenvalues,
\href{https://arxiv.org/abs/2405.16354}{arXiv:2405.16354}, 2024;
comparisons use version~2.\par}

\bibitem{FLPSBalls2023}
N. Filonov, M. Levitin, I. Polterovich and D.~A. Sher,
P\'olya's conjecture for Euclidean balls,
\emph{Invent. Math.} \textbf{234} (2023), 129--169.
\href{https://doi.org/10.1007/s00222-023-01198-1}{doi:10.1007/s00222-023-01198-1}.

%\Needspace{6\baselineskip}
\bibitem{HeWang2026}
X. He and Z. Wang,
P\'olya's conjecture for thin products,
\emph{Int. Math. Res. Not.} \textbf{2026} (2026), no.~16, rnag182.
\href{https://doi.org/10.1093/imrn/rnag182}{doi:10.1093/imrn/rnag182}.
Accessible manuscript:
\href{https://arxiv.org/abs/2402.12093}{arXiv:2402.12093}.

\bibitem{JiangLin2026}
R. Jiang and F. Lin,
P\'olya's conjecture up to $\varepsilon$-loss and quantitative estimates
for the remainder of Weyl's law,
\emph{Comm. Pure Appl. Math.} (2026), e70058, Early View.
\href{https://doi.org/10.1002/cpa.70058}{doi:10.1002/cpa.70058}.
Preprint: \href{https://arxiv.org/abs/2507.04307v4}{arXiv:2507.04307v4}.
Theorem and equation numbers used here refer to this version.



\bibitem{DangWangXieZhou2023}
H. Dang, Y. Wang, H. Xie and C. Zhou,
Enhanced error estimates for augmented subspace method,
\emph{J. Sci. Comput.} \textbf{94} (2023), no.~2, Article~40.
\href{https://doi.org/10.1007/s10915-022-02090-5}
{doi:10.1007/s10915-022-02090-5}.

\bibitem{XieZhangOwhadi2019}
H. Xie, L. Zhang and H. Owhadi,
Fast eigenpairs computation with operator adapted wavelets and
hierarchical subspace correction,
\emph{SIAM J. Numer. Anal.} \textbf{57} (2019), no.~6, 2519--2550.
\href{https://doi.org/10.1137/18M1194079}
{doi:10.1137/18M1194079}.
\end{thebibliography}
\end{document}